\documentclass{amsart}

\newtheorem{theorem}{Theorem}[section]

\theoremstyle{definition}

\newtheorem{example}[theorem]{Example}

\theoremstyle{remark}
\newtheorem{remark}[theorem]{Remark}

\numberwithin{equation}{section}

\begin{document}

\title{ The Bloch groups and special values of Dedekind zeta functions}

%    Information for first author
\author{Chaochao Sun}
%    Address of record for the research reported here
\address{ Corresponding author: Chaochao Sun }
\address{Chaochao Sun, School of  Mathematics and Statistics, Linyi University, Linyi,
China 276005}
%    Current address

\email{sunuso@163.com}
\author{Long Zhang}
%    Address of record for the research reported here
\address{Long Zhang, School of Mathematics and Statistics, Qingdao University, Qingdao, China, 266071}
\email{zhanglong\_note@hotmail.com}
%    \thanks will become a 1st page footnote.
\thanks{The research was supported  by NNSFC Grant \#11601211.}

%    Information for second author

%\thanks{Support information for the second author.}

%    General info

%\date{January 1, 2001 and, in revised form, June 22, 2001.}

%\dedicatory{This paper is dedicated to our advisors.}


\begin{abstract}
In this paper, we compare the two definitions of Bloch group, and survey the elements in Bloch group. We confirm the Lichtenbaum conjecture on the field $\mathbb{Q}(\zeta_p)$ under the assumption the truth of the base of the Bloch group of $\mathbb{Q}(\zeta_p)$ and the relations of $K_2$ group. We also study the Lichtenbaum conjecture on non-Galois fields. By PARI, we get some equations of  the zeta functions on special values and the structure of tame kernel of these fields.
\end{abstract}

\maketitle

%% The correct journal style for \specialsection is all uppercase; a known bug
%% in amsart.cls prevents this, so input must be uppercase until it is fixed.
%\specialsection*{This is a Special Section Head}
%%%%%%%%%%%%%%%%%%%%%%%%%%%%%%%%%%%%%%%%%%%%%%%%%%%%%%%%%%%%%%%%%%%%%%%%

%%%%%%%%%%%%%%%%%%%%%%%%%%%%%%%%%%%%%%%%%%%%%%%%%%%%%%%%%%%%%%%%%%%%%%%%

\section{Introduction}

The special values of zeta function of number fields are an interesting field in number theory. When we consider the residue of zeta function,  there is a class number formula as following
\begin{equation}
R_1=-\frac{w}{h}\lim_{s\to 0}\zeta_F(s)s^{-(r_1+r_2-1)}
\end{equation}
where $R_1$ is the Dirichlet regulator of the field $F$, $\zeta_F(s)$ is the Dedekind's zeta function, $w$ is the root number of unity and $h$ is the class number.

In order to generalize the formula (1.1) to higher K-theory, Borel\cite{Bo} has introduced morphisms
$$r: K_{2m-1}(\mathfrak{o}_F)\to V_m$$
where $m\geq 2$, $\mathfrak{o}_F$ is the algebraic integer of number field $F$, $V_m$ is a real vector space of dimension
\[
\dim_{\mathbb{R}}V_m=d_m=
\begin{cases}
r_1+r_2  & \mbox{if}\ m \ \mbox{is odd}, m>0,\\
r_2   & \mbox{if}\ m \ \mbox{is even}
\end{cases}
\]
where $r_1$(respectively $r_2$) is the number of real(respectively complex) places of $F$.
Borel has proved that $r(K_{2m-1}(\mathfrak{o}_F))$ is a lattice of $V_m$, so the rank of $K_{2m-1}(\mathfrak{o}_F)$ is $d_m$.
Let $R_m(F)$ be a twisted version of the $m$th Borel regulator(see \cite{Bo2}), the twisted regulator map $r_m$ being a map

$$r_m: K_{2m-1}(\mathfrak{o}_F)\to [(2\pi i)^{m-1}\mathbb{R}]^{d_m} $$
Borel proved that, up to a rational factor, $R_m(F)$ is equal to $\zeta^*_F(1-m)$, the first non-vanishing Taylor coefficient of $\zeta_F (s)$ at $s=1-m$.
 Lichtenbaum's conjecture\cite{Li}(as modified by Borel \cite{Bo2}), tries to interpret this rational factor and asks whether for all number fields and for any integer $m\geq2$ there is a relation of the form
$$\lim_{s\to 1-m}\zeta_F (s)(s-1+m)^{-d_m}=\pm\frac{\#K_{2m-2}(\mathfrak{o}_F)}{\#K_{2p-1}^{\text{ind}}(\mathfrak{o}_F)_{\text{tor}}}\cdot R_m(F).$$
For $m=2$ and $F$ totally-real abelian it has been proved (up to a power of 2) by Mazur and Wiles \cite{MW} as a consequence of their proof of the main conjecture of Iwasawa theory. In\cite{KNF}, Kolster, Nguyen Quang Do and Fleckinger  have proved a modified version of the conjecture (also up to a power of 2) for all abelian fields $F$ and $m\geq 2$.

When $m=2$, Bloch \cite{Bl} suggested and D.Burns, R.de Jeu, H.Gangl\cite{DB} finally proved that Borel's regulator map can be given  in terms of the Bloch-Wigner dilogarithm $D_2(z)$ as a map on the Bloch group $B(F)$. While for Bloch group, there exist two kinds of definitions, see \cite{B.J2} and \cite{B.J3}. In section 1, we compare the two definitions of Bloch of number field(Theorem2.1)and discuss the relations of elements in Bloch group. Let $\widetilde{R}_2(F)$ be the second dilogarithmic regulator(see \cite{B.J3}), $w_2(F)$ be the number of roots of unity in the compositum of all quadratic extensions of $F$. Then the Lichtenbaum conjecture can be read as follows
\begin{equation}
|\zeta_F^*(-1)|=\frac{\widetilde{R}_2(F)\cdot\#K_2(\mathfrak{o}_F)}{w_2(F)}
\end{equation}
In section 2, assume two conjectures, we can prove the above version Lichtenbaum conjecture on the field $\mathbb{Q}(\zeta_p)$. In sectoin 3, we get some fields which are not Galois. On these fields, we get some functional equations on zeta functions and the dilogarithm functions when $s=2$ by comparing the numerical results. Using PARI, we get the structures of $\#K_2(\mathfrak{o}_F)$.

\section{The Blobh group}

Let $F$ be a field of char$(F)=0$, $\mathbb{Z}[F]:=\bigoplus\limits_{1\neq a\in F^{\times}
 }\mathbb{Z}[a]$ be a free abelian group generated by base $[a]$.
We have a natural  map
\begin{align*}
 \partial : \mathbb{Z}[F] &\to F^{\times}\wedge F^{\times}:=F^{\times}\otimes F^{\times}/\langle a\otimes (-a)\rangle\\
 [a]&\mapsto \overline{a\otimes (1-a)}
\end{align*}

Let $H$ be the subgroup generated by the elements $[a]+[1-a],[a]+[a^{-1}],[a]+[b]+[\frac{1-a}{1-ab}]+[1-ab]+[\frac{1-b}{1-ab}] $,
$a,b\in F^{\times}\setminus \{1\},ab\neq1$. It is easy to check that $H\subseteq \ker\partial$ and the Bloch group of $F$ is defined to be
$$ B(F):=\ker\partial/H$$

In Suslin's paper\cite{S.A}, the Bloch group is defined by another form, we state it as follows.
Let $\varphi$ be a map
\begin{align*}
 \varphi: \mathbb{Z}[F] &\to F^{\times}\otimes F^{\times}/\langle x\otimes y +y\otimes x\rangle\\
 [a]&\mapsto \overline{a\otimes (1-a)}
\end{align*}
First, we have(see \cite{S.A})
 $$\varphi([x]-[y]+[\frac{y}{x}]-[\frac{1-x^{-1}}{1-y^{-1}}]+[\frac{1-x}{1-y}])=\overline{x\otimes(\frac{1-x}{1-y})+(\frac{1-x}{1-y})\otimes x}=0.$$
Let $I$ be the subgroup of $\mathbb{Z}[F]$ generated by the elements$[x]-[y]+[\frac{y}{x}]-[\frac{1-x^{-1}}{1-y^{-1}}]+[\frac{1-x}{1-y}]$. Then
the Bloch group in Suslin's paper is defined by
$$\mathcal{B}(F):=\ker\varphi/I$$
Although the definitions of Bloch group are a little different,in fact, they differ at most by torsion. We have the following result
\begin{theorem}  The two kinds of Bloch groups $B(F)$ and $\mathcal{B}(F)$ are different by torsion. Furthermore,
$$B(F)\otimes\mathbb{Z}[\frac{1}{6}]\cong \mathcal{B}(F)\otimes\mathbb{Z}[\frac{1}{6}].$$
\end{theorem}

\begin{proof}
Since
$$ x\otimes y +y\otimes x= xy\otimes(-xy) -x\otimes(-x)-y\otimes(-y)$$
$$2 (a\otimes(-a))=2 (a\otimes (-1)) +2 (a\otimes a)=a\otimes a+a\otimes a,$$
we have
$$\langle x\otimes y +y\otimes x\rangle \subseteq \langle a\otimes(-a)\rangle,2\langle a\otimes(-a)\rangle\subseteq \langle x\otimes y +y\otimes x\rangle,$$
where $a,x,y\in F^{\times}$. So, there exists the following inclusions
\begin{equation}\ker\varphi \subseteq \ker\partial, 2\ker\partial\subseteq\ker\varphi.\end{equation}

Another, we have $I\subseteq H$,we can show the inclusion from the generator:
\begin{multline}
[x]-[y]+[\frac{y}{x}]-[\frac{1-x^{-1}}{1-y^{-1}}]+[\frac{1-x}{1-y}]\\
=([y^{-1}]+[yx^{-1}]+[\frac{1-y^{-1}}{1-x^{-1}}]+[1-x^{-1}]+[\frac{y-x}{1-x}])+([x]+[x^{-1}])+([\frac{1-x}{1-y}]+[\frac{1-y}{1-x}])\\
-([y]+[y^{-1}])-([\frac{1-x^{-1}}{1-y^{-1}}]+[\frac{1-y^{-1}}{1-x^{-1}}])-([1-x^{-1}]+[x^{-1}])-([1-\frac{1-y}{1-x}]+[\frac{1-y}{1-x}]).
\end{multline}

By results in \cite{S.A},we have
\begin{gather}
2([x]+[x^{-1}]),6([x]+[1-x])\in I.
\end{gather}
So, by (2.2),(2.3), we have $6H\subset I$.

Then there is an exact sequence
\begin{equation}
0\to(\ker\varphi\cap H)/I\to\ker\varphi/I=\mathcal{B}(F)\to B(F)=\ker\partial/H \to\ker\partial/\ker\varphi\to 0.\end{equation}

 Because $6(\ker\varphi\cap H)/I =0, 2(\ker\partial/\ker\varphi)=0$, we obtain that $B(k)$ and $\mathcal{B}(k)$ are different by torsion. Tensor with $\mathbb{Z}[\frac{1}{6}]$ on (2.4), the flatness of $\mathbb{Z}[\frac{1}{6}]$ leads to get the isomorphism
$$B(F)\otimes\mathbb{Z}[\frac{1}{6}]\cong \mathcal{B}(F)\otimes\mathbb{Z}[\frac{1}{6}].$$
\end{proof}
\bigskip

The Bloch group are related with zeta function by  Bloch-Wigner function.
Now let us introduce  the Bloch-Wigner function:
$$D(z)=-\text{Im}\int_0^z\frac{\log(1-t)}{t}dt+\arg (1-z)\cdot \log|z|, \ z\in \mathbb{C}.$$
$D(z)$ is real analytic on $\mathbb{C}\setminus \{0, 1 \}$ and continuous at $0, 1 $. It satisfies that
$$D(\bar{z})=-D(z), D(z)=-D(1-z)=-D(z^{-1})$$
$$D(z_1)+D(z_2)+D(\frac{1-z_1}{1-z_1z_2})+D(1-z_1z_2)+D(\frac{1-z_2}{1-z_1z_2})=0$$
where $z,z_1,z_2 \in\mathbb{C}',z_1z_2\neq1$.

\begin{example} \quad Let $\zeta_p=e^{\frac{2\pi i}{p}}$. Then in the field $\mathbb{Q}(\zeta_3)$, $\zeta_3=-\overline{\zeta}_6$.
So, we have
\begin{equation}
D(\zeta_3)=\text{Im}(\sum_{n=1}^{\infty}\frac{\zeta_3^n}{n^2})
=\frac{\sqrt{3}}{2}\sum_{n=1}^{\infty}\frac{\chi_3(n)}{n^2}
=\frac{\sqrt{3}}{2}L(\chi_3,2)
\end{equation}
where $\chi_3$ is the primitive Dirichlet character with conductor 3.

Similarly, we have
\begin{equation}D(\zeta_6)=\frac{\sqrt{3}}{2}(L(\chi_6,2)+\frac{1}{4}L(\chi_3,2))\end{equation}
where  $\chi_6$ is the primitive Dirichlet character with conductor 6.

On the other hand, we have
\begin{align*}
L(\chi_3,2)&=1-\frac{1}{2^2}+\frac{1}{4^2}-\frac{1}{5^2}+\frac{1}{7^2}-\frac{1}{8^2}+\cdots\\
&=(1-\frac{1}{5^2}+\frac{1}{7^2}-\frac{1}{11^2}+\cdots)-(\frac{1}{2^2}-\frac{1}{4^2}+\frac{1}{8^2}-\frac{1}{10^2}+\cdots)\\
&=L(\chi_6,2)-\frac{1}{4}L(\chi_3,2)
\end{align*}
So, we have
\begin{equation}L(\chi_6,2)/L(\chi_3,2)=\frac{5}{4}.\end{equation}

 By(2.5),(2.6) and (2.7), we have $D(\zeta_6)/D(\zeta_3)=\frac{3}{2}$. This relation can be reflected onto the elements of Bloch group.
In fact, we have the following result

\bigskip
{\bf Claim 1}\quad $2[\zeta_6]=3[\zeta_3]\in B(\mathbb{Q}(\zeta_3))$

\begin{proof}\quad Suppose $\zeta_n \in F$, then we have $[x^n]=n([x]+[\zeta_nx]+\cdots+[\zeta_n^{n-1}x])$ in $B(F)$.
That's because
\begin{align*}
x^n\otimes(1-x^n)&=n(x\otimes((1-x)\cdots(1-\zeta_n^{n-1}x)))\\
&=n(x\otimes(1-x)+\cdots +x\otimes(1-\zeta_n^{n-1}x))\\
&=n(x\otimes(1-x)+\cdots +\zeta_n^{n-1}x\otimes(1-\zeta_n^{n-1}x)).
\end{align*}
Now,let $\zeta_n=-1, x=\zeta_6\in \mathbb{Q}(\zeta_3)$. Then in $B(\mathbb{Q}(\zeta_3))$ we get
$$2[\zeta_6]=[\zeta_3]-2[-\zeta_6]=[\zeta_3]-2[\zeta_3^{-1}]=3[\zeta_3]. $$
\end{proof}
\end{example}

\begin{example} \quad Let $\zeta_5=\exp(2\pi i/5), x_1=1+\zeta_5+\zeta_5^2, x_2=-\zeta_5^4$. In \cite{B.J1}, Browkin found that
 $b_1:=2[x_1]+4[x_2]\in B(\mathbb{Q}(\zeta_5))$. Let $\sigma$ be a automorphism of $ B(\mathbb{Q}(\zeta_5))$ such that
 $\sigma(\zeta_5)=\zeta_5^2$. Then $b_2:=2[\sigma(x_1)]+4[\sigma(x_2)]\in B(\mathbb{Q}(\zeta_5))$.
Let $a_1=5[\zeta_5], a_2=5[\zeta_5^2]\in B(\mathbb{Q}(\zeta_5))$.
In fact, we have $a_1=b_1, \ a_2=b_2$ in $B(\mathbb{Q}(\zeta_5))$. Now, we show how to get it.
Because $(1+\zeta_5+\zeta_5^2)(1+\zeta_5^3)=1$, we get
\begin{align*}
b_1&=2[\frac{1}{1+\zeta_5^3}]+4[-\zeta_5^4]\\
&=-2[1+\zeta_5^3]+2[\zeta_5^3]-4[\zeta_5^4]\\
&=2[-\zeta_5^3]+2[\zeta_5^3]-4[\zeta_5^4]\\
&=[\zeta_5^6]+4[\zeta_5]=5[\zeta_5]\\
&=a_1
\end{align*}
\end{example}

Moreover, $\zeta_F^*(-1)=0.0248111839$. By the Lichtenbaum conjecture:
\begin{equation*}
|\zeta_F^*(-1)|=\frac{\widetilde{R}_2(F)\cdot\#K_2(\mathfrak{o}_F)}{w_2(F)}.
\end{equation*}
It is easy to see that $w_2(F)=120$. Assuming dilogarithmic lattice $\Lambda$ in $\mathbb{R}^2$ generated by the vectors
$$(\widetilde{D}(b_1),\widetilde{D}(\sigma(b_1)))\quad \text{and}\quad (\widetilde{D}(b_2),\widetilde{D}(\sigma(b_2)))$$
is full lattice, where $\widetilde{D}(z)=\frac{1}{\pi}D(z)$. Substituting the above numerical data we get $\#K_2(O_F)=1$. It is proven in \cite{Zh} that $\#K_2(O_F)=1$.
\bigskip

\section{Special value of zeta function of $\mathbb{Q}(\zeta_p)$ }

Now, we want to study the Lichtenbaum conjecture on the special value of zeta function at $-1$ in the case $F=\mathbb{Q}(\zeta_p)$,  where $p$ be an odd prime number, $\zeta_p$ be a  primitive root of unity. It is easy to see that $p[\zeta_p^i]\in B(\mathbb{Q}(\zeta_p)), i=1,\cdots, \frac{p-1}{2}$. The subgroup $Z$ generated by  $p[\zeta_p],p[\zeta_p^2],\cdots,p[\zeta_p^{\frac{p-1}{2}}]$ is a finite index subgroup of the Bloch group $B(\mathbb{Q}(\zeta_p))$. The covolume of $Z$ in $\mathbb{R}^{\frac{p-1}{2}}$ under the regulator map $D$ is denoted by $R$.

\begin{theorem}
The covolume of $Z$ is
$$R=(2\pi)^{\frac{1-p}{2}}p^{\frac{3(p-1)}{4}}\prod_{\chi \text{odd char.}}|L(\chi,2)|,$$
where  $\chi$ is the character of the group $(\mathbb{Z}/p)^*$.
\end{theorem}
\begin{proof}
According to \cite{KNF}P.713, we have
\begin{equation}
R=|\det \alpha|=\prod_{\chi \text{odd char.}}|\sum_{a\in(\mathbb{Z}/p)^*}\frac{p\chi(a)D_2(e^{2\pi ia/p})}{2\pi}|,
\end{equation}
where $D_2(z)=\sum_{n\geq1}\frac{z^n}{n^2}$, $\chi$ is the character of the group $(\mathbb{Z}/p)^*$.

On the other hand, by \cite{Gr}P.12, we get
\begin{equation}
|L(2,\chi)|=|g(\chi)| \cdot |\sum_{a\in(\mathbb{Z}/p)^*}\chi^{-1}(a)D_2(e^{-2\pi ia/p})|.
\end{equation}

 By the definition of $g(\chi)$ in \cite{Gr} and the property of Gauss sum, we get$|g(\chi^{-1})|=p^{-\frac{1}{2}}$, so from (3.2)we have
\begin{equation}
|L(2,\chi^{-1})|=p^{-\frac{1}{2}}|\sum_{a\in(\mathbb{Z}/p)^*}\chi(a)D_2(e^{2\pi ia/p})|.
\end{equation}

Combining (3.1) and (3.3), we get
$$R=(2\pi)^{\frac{1-p}{2}}p^{\frac{3(p-1)}{4}}\prod_{\chi \text{odd char.}}|L(\chi,2)|.$$
\end{proof}

{\bf Conjecture 1}\quad  $p[\zeta_p^i]\in B(\mathbb{Q}(\zeta_p)), i=1,\cdots, \frac{p-1}{2}$ is the free part of Bloch group of $B(\mathbb{Q}(\zeta_p))$.

Another conjecture is related with the $K_2$ groups of the algebraic integers of $\mathbb{Q}(\zeta_p)$ and $\mathbb{Q}(\zeta_p)^+$, which denotes the maximal real subfield of $\mathbb{Q}(\zeta_p)$. Then we have the following conjecture.

{\bf Conjecture 2}\quad There is a natural exact sequence of $F=\mathbb{Q}(\zeta_p)$
$$0\to \ker \psi \to K_2(\mathcal{O}_{F^+})\xrightarrow{\psi}{} K_2(\mathcal{O}_{F})\to 0,$$
and the order of $\ker \psi$ is $2^{\frac{p-1}{2}}$.

\begin{remark} Conjecture 2 is true for $p=3,5$. For $p=3$, it is easy to check that this case is true. For $p=5$, by \cite{Zh} we know $K_2(\mathcal{O}_{F})=0$. Since $\#K_2(\mathcal{O}_{F^+})=4$, we get Conjecture 2 holds for $p=5.$
\end{remark}

\begin{theorem}  Assuming the above two conjecture,  the zeta function of $F=\mathbb{Q}(\zeta_p)$ has the following equation
$$|\zeta_{F}^*(-1)|=\frac{\#K_2(\mathcal{O}_{F})}{w_2(F)}\widetilde{R}_2(F)$$
\end{theorem}

\begin{proof}
In \cite{Bl}, Bloch has calculated that
\begin{equation}|\det(D(\sigma_i(\zeta_p^j)))|=2^{\frac{1-p}{2}}p^{\frac{p-1}{4}}\prod_{\chi \text{odd char.}}|L(\chi,2)|,\end{equation}
where $\sigma_i:=\sigma^i, \sigma(\zeta_p)=\zeta_p^{2}$,  $\chi$ runs all the odd character of $\mathbb{F}_p^{\times}$.
Browkin defined $\widetilde{R}_2(F)=\frac{1}{\pi^{r_2}}|\det(D(\sigma_i(\varepsilon_j)))|$ to be the second regulator, where $\sigma_i$
is the complex places, $\varepsilon_j$ is the base of Bloch group $B(F)$,$1\leq i, j\leq r_2$. By Conjecture 1,  $\widetilde{R}_2(F)=R$  is the second regulator of $F=\mathbb{Q}(\zeta_p)$. So, we have
\begin{equation}
\widetilde{R}_2(F)=(2\pi)^{\frac{1-p}{2}}p^{\frac{3(p-1)}{4}}\prod_{\chi \text{odd char.}}|L(\chi,2)|
\end{equation}

The absolute value of discriminant of $F$ is $|d_F|=p^{p-2}$.  Decomposing the zeta function $\zeta_F(s)$
into the Dirichlet $L$-function, and taking use of the fact $\zeta_{F^+}(s)=\prod_{\chi \text{even char.}}L(\chi,s)$, we have
\begin{equation}\zeta_F(s)=\zeta_{F^+}(s)\cdot \prod_{\chi \text{odd char.}}L(\chi,s).
\end{equation}

Using (3.5),(3.6), $\Gamma^*(-1)=1$ and the function equation of $\zeta_F(s)$, we get that
\begin{equation}|\zeta_F^*(-1)|=2^{1-p}\pi^{1-p}p^{\frac{3p-9}{4}}\zeta_{F^+}(2)\widetilde{R}_2(F)
\end{equation}

Using the function equation of $\zeta_{F^+}(s)$ and $|\Gamma(-\frac{1}{2})|=2\pi^{\frac{1}{2}}$, we have
\begin{equation}\zeta_{F^+}(2)=2^{\frac{p-1}{2}}\pi^{p-1}|d_{F^+}|^{-\frac{3}{2}}|\zeta_{F^+}(-1)|,
\end{equation}
where $d_{F^+}$ is the discriminant of $F^+$.

So, combining (3.7) and (3.8) , we have
\begin{equation}|\zeta_F^*(-1)|=2^{\frac{1-p}{2}}p^{\frac{3p-9}{4}}|d_{F^+}|^{-\frac{3}{2}}|\zeta_{F^+}(-1)|\widetilde{R}_2(F).
\end{equation}

In fact, By the Theorem 3.11 in \cite{W.L}, we obtain that $d_{F^+}=p^{\frac{p-3}{2}}$, hence, $|d_{F^+}|^{\frac{3}{2}}=p^{\frac{3p-9}{4}}$. So, from (3.9), we get
\begin{equation}|\zeta_F^*(-1)|=2^{\frac{1-p}{2}}|\zeta_{F^+}(-1)|\widetilde{R}_2(F).
\end{equation}

Wiles\cite{Wi} has proved that the Birch-Tate conjecture is true for the abelian totally real field. So, for $F^+$ we have
\begin{equation}|\zeta_{F^+}(-1)|=\frac{\#K_2(\mathcal{O}_{F^+})}{w_2(F^+)}
\end{equation}
The method of calculating the number $w_2(F)$ can be found in Weibel paper \cite{W.C}. For $\mathbb{Q}(\zeta_p)$, we get that
\begin{equation}
w_2(F)=w_2(F^+)=
\begin{cases}
24,\ \  p=3;\\
24p,\ \ p\neq3.
\end{cases}
\end{equation}
Hence, from (3.10),(3.11) and (3.12), we get
$$|\zeta_{F}^*(-1)|=2^{\frac{1-p}{2}}\frac{\#K_2(\mathcal{O}_{F^+})}{w_2(F)}\widetilde{R}_2(F)$$

By Conjecture 2, we know the  Lichtenbaum conjecture is true for $\mathbb{Q}(\zeta_p)$.
\end{proof}

\textbf{Remark3.4} Professor T.Nguyen Quang Do has recently told us that the Lichtenbaum conjecture has now been proved in full generality for abelian fields (see the literature
\cite{C} Chapter 9). We are grateful for his account of the status of the Lichtenbaum conjecture.

\bigskip

\section{Lichtenbaum conjecture on non-Galois fields}

\bigskip

Suppose $F$ is a number field with $r_2(F)=1$, then the free part of $B(F)$ is $\mathbb{Z}$ module of rank 1. In this section, we list some fields $F$ with $r_2(F)=1$. The elements of  Bloch group  $B(F)$ are constructed in a flexible way. Assume the base of Bloch group and the Lichtenbaum conjecture, we get a conjectural order of the $K_2$ group of $\mathcal{O}_F$.

 In \cite{Go}P.250, there is a theorem about the special value of Dedekind zeta function at 2 and the Borel regulator,which states as following
 \begin{theorem}  Let $\zeta_F(s)$ be the Dedekind zeta function of $F$. Then there exist
 $$ y_1,\cdots,y_{r_2}\in B(F)$$
 such that
 $$\zeta_F(2)=q\cdot \pi^{2(r_1+r_2)}\cdot|d_F|^{1/2}\cdot\det|D(\sigma_{r_1+j}(y_i))|$$
 where $1\leq i,j\leq r_2$ and $q$ is some rational number.
 \end{theorem}

 Using PARI, we find the equations between the special values of zeta function at $2$ and the  dilogarithm functions. These equations are expected to be given a proof.
Using the computing programm in \cite{BH}, we compute that all the K-groups are confirmed with the conjectural order and give their structures.

\begin{example} Consider the equations as follows
$$
\begin{cases}
1+y=x\\
1-y^{-1}=-x^4y^{-1}.
\end{cases}$$

We get that $x^3+x^2+x+2=0$. Let $\alpha$ be a root of this equation. Then $r_2(\mathbb{Q}(\alpha))=1$ and we can assume $\mathbb{Q}(\alpha)$ is complex. Now we claim that
$-4[\alpha]-[\alpha-1]\in B(\mathbb{Q}(\alpha))$. First, let $\beta=\alpha-1$, then we have $-4[\alpha]-[\alpha-1]=4[-\beta]+[\beta^{-1}]$. So,
\begin{align*}
\partial(4[-\beta]+[\beta^{-1}])&=4[(-\beta)\wedge(1+\beta)]+\beta^{-1}\wedge(1-\beta^{-1})\\
&=(-\beta)^{4}\wedge\alpha+\beta^{-1}\wedge(-\alpha^4\beta^{-1})\\
&=\beta^4\wedge\alpha+\beta^{-1}\wedge\alpha^4+\beta^{-1}\wedge(-\beta^{-1})\\
&=4\beta\wedge\alpha-4\beta\wedge\alpha\\
&=0.
\end{align*}
Hence, $-4[\alpha]-[\alpha-1]\in B(\mathbb{Q}(\alpha))$.

Assuming the Lichtenbaum conjecture and $-4[\alpha]-[\alpha-1]\in B(\mathbb{Q}(\alpha))$ being a base, we get the order of $K_2(\mathcal{O}_{\mathbb{Q}(\alpha)})$, i.e.
$\#K_2(\mathcal{O}_{\mathbb{Q}(\alpha)})=4$. In fact, let $F=\mathbb{Q}(\alpha),\theta=4[\alpha]+[\alpha-1]\in B(F)$. Using PARI we have
$$\alpha\approx0.176604982099662+1.202820819285479i,$$
$$R_2(F):=|D(\theta)|=D(\theta)\approx4.415332477453866,$$
$$\zeta_{F}(2)\approx1.516751720642021.$$
Because $r_1(F)=r_2(F)=1, |\Gamma^*(-1)|=1, |\Gamma(-\frac{1}{2})|=2\sqrt{\pi}$, by the function equation of $\zeta_F(s)$, we have
$$\zeta_F(2)=2^4\pi^5|d_F|^{-\frac{3}{2}}\zeta_F^*(-1).$$
Using the Lichtenbaum conjecture on $F=\mathbb{Q}(\alpha)$, we get
$$\zeta_F^*(-1)=\frac{\#K_2(\mathcal{O}_F)R_2(F)}{w_2(F)\pi}.$$
By Proposition 20,22 in \cite{W.C}, we know $w_2(F)=24$; By PARI, we find $d_F=-83$. At last, we get
$$\#K_2(\mathcal{O}_F)=\frac{3\cdot83^{\frac{3}{2}}\zeta_F(2)}{2\pi^4R_2}.$$
Using PARI, we get
$$\#K_2(\mathcal{O}_F)=3.999999999999999=4.$$
Using method in \cite{BH}, we have $K_2(\mathcal{O}_F)$ is isomorphic to $\mathbb{Z}/2\times\mathbb{Z}/2$.

 Moreover, we get the equation as follows by the numerical method
 $$\zeta_F(2)=\frac{8\pi^4}{3\cdot83^{\frac{3}{2}}}D(\theta).$$
\end{example}

\begin{example}  Consider the equations as follows
$$
\begin{cases}
1+y=x\\
1-y^{-1}=x^{2}
\end{cases}.$$
We get that $x^3-x^2-x+2=0$. Let $\alpha$ be its complex root and $F=\mathbb{Q}(\alpha)$. Then we know $r_1(F)=r_2(F)=1$ and we claim that
$-[\alpha^2]-2[\alpha]\in B(F)$. Let $\beta=\alpha-1$.
Then
\begin{align*}
\partial(-[\alpha^2]-2[\alpha])&=\partial([1-\alpha^2]+2[1-\alpha])\\
&=\partial([\beta^{-1}]+2[-\beta])\\
&=\beta^{-1}\wedge(1-\beta^{-1})+2[(-\beta)\wedge(1+\beta)]\\
&=\beta^{-1}\wedge\alpha^{2}+2[(-\beta)\wedge\alpha]\\
&=-2(\beta\wedge\alpha)+2(\beta\wedge\alpha)\\
&=0.
\end{align*}

Let $\theta=[\alpha^2]+2[\alpha]$. Computing by PARI, we get that
$$\alpha\approx1.102784715200295+0.665456951152813i,$$
$$D(\theta)\approx2.568970600936709,$$
$$\zeta_F(2)\approx1.472479780199297,$$
$$d_F=-59 .$$
And we have $w_2(F)=24$. Assuming the Lichtenbaum conjecture and $\theta$ the base of $B(F)$, then we have
$$\#K_2(\mathcal{O}_F)=\frac{3\cdot59^{\frac{3}{2}}\zeta_F(2)}{2\pi^4R_2}.$$
By PARI we have $\#K_2(\mathcal{O}_F)=4.000000000000000=4$. Using method in \cite{BH}, we have $K_2(\mathcal{O}_F)$ is isomorphic to $\mathbb{Z}/2\times\mathbb{Z}/2$. The equation of zeta function at 2 is
$$\zeta_F(2)=\frac{8\pi^4}{3\cdot59^{\frac{3}{2}}}D(\theta).$$
\end{example}

\begin{example}
 Consider the equations as follows
$$
\begin{cases}
1+y=x\\
1-y^{-1}=y^{-1}x^3
\end{cases}.$$
We have $x^3-x+2=0$. Let $\alpha$ be complex root of it, we have $r_2(\mathbb{Q(\alpha)})=1$.
Then, $\theta:=-6[1-\alpha]-2[\frac{1}{\alpha-1}]\in B(\mathbb{Q(\alpha)})$. In fact, let $\beta=\alpha-1.$ Then we have
\begin{align*}
\partial(-6[1-\alpha]-2[\frac{1}{\alpha-1}])&=\partial(-6[1-\alpha]-2[\beta^{-1}])\\
&=-6(-\beta\wedge\alpha)-2(\beta^{-1}\wedge\beta^{-1}\alpha^3)\\
&=-6(\beta\wedge\alpha)+6(\beta\wedge\alpha)\\
&=0.
\end{align*}

 By PARI, we obtain
$$\alpha\approx0.760689853402284+0.857873626595179i,$$
$$D(\theta)\approx7.517689896474569,$$
$$\zeta_F(2)\approx1.841207016617394,$$
$$d_F=-104.$$
Assuming Lichtenbaum conjecture and $\theta$ being the base of $B(F)$, since $w_2(F)=24$, we have
$$\#K_2(\mathcal{O}_F)=\frac{3\cdot104^{\frac{3}{2}}\zeta_F(2)}{2\pi^4D(\theta)}=4.000000000000000=4.$$
Using method in \cite{BH}, we have $K_2(\mathcal{O}_F)$ is isomorphic to $\mathbb{Z}/2\times\mathbb{Z}/2$ and the function equation is
$$\zeta_F(2)=\frac{8\pi^4}{3\cdot104^{\frac{3}{2}}}D(\theta).$$
\end{example}

\begin{example}
 Consider the following equations
$$
\begin{cases}
1-y=x^2\\
1-y^{-1}=y^{-3}x^{3}.
\end{cases}$$
It is easy to get the equation $x^4-2x^2+x+1=0.$ Let $\alpha$ be a complex root of this equation. Then we have $r_2(\mathbb{Q(\alpha)})=1$.
Let $\beta=1-\alpha^2$. Then we have
$$\partial(3[\beta]+2[\beta^{-1}])
=3\beta\wedge\alpha^2+2\beta^{-1}\wedge\beta^{-3}\alpha^3=0.$$
Hence, $[\beta]=3[\beta]+2[\beta^{-1}]\in B(F)$, that is, $[\alpha^2]\in B(F)$. Using PARI, we have
$$\alpha\approx1.007552359378179+0.513115795597015i$$
$$D(\alpha^2)\approx0.981368828892232$$
$$\zeta_F(2)\approx1.056940574599707$$
$$d_F=-283$$
Assuming Lichtenbaum conjecture and $[\alpha^2]$ being the base of $B(F)$, since $w_2(F)=24$, we have
$$\#K_2(\mathcal{O}_F)=\frac{3\cdot283^{\frac{3}{2}}\zeta_F(2)}{2^2\pi^6R_2}.$$
By PARI, we have $\#K_2(\mathcal{O}_F)=3.99999999999999=4$. Using method in \cite{BH}, we have $K_2(\mathcal{O}_F)$ is isomorphic to $\mathbb{Z}/2\times\mathbb{Z}/2$ and the function equation is
$$\zeta_F(2)=\frac{16\pi^6}{3\cdot283^{\frac{3}{2}}}D(\alpha^2).$$
\end{example}

\begin{example}
 Considering the equations as follows
$$
\begin{cases}
1-y=-x\\
1-y^{-1}=x^4,
\end{cases}$$
we have $x^4+x^3-1=0$. Let $\alpha$ be complex root of it, we have $r_2(\mathbb{Q(\alpha)})=1$
and $3[-\alpha]\in B(\mathbb{Q(\alpha)})$. By PARI, we have
$$\alpha\approx-0.219447472149275-0.914473662967726i,$$
$$R_2:\approx3D(-\alpha)=2.944106486676696,$$
$$\zeta_F(2)\approx1.056940574599707,$$
$$d_F=-283.$$
Assuming Lichtenbaum conjecture and $3[-\alpha]$ being the base of $B(F)$, since $w_2(F)=24$, we have
$$\#K_2(\mathcal{O}_F)=\frac{3^2\cdot283^{\frac{3}{2}}\zeta_F(2)}{2^2\pi^6R_2}=3.99999999999999=4.$$
Using method in \cite{BH}, we have $K_2(\mathcal{O}_F)$ is isomorphic to $\mathbb{Z}/2\times\mathbb{Z}/2$ and function equation is
$$\zeta_F(2)=\frac{16\pi^6}{3\cdot283^{\frac{3}{2}}}D(-\alpha).$$
\end{example}

\bigskip

\bibliographystyle{amsplain}

\begin{thebibliography}{10}
\bibitem{BH}K. Belabas,  H. Gangl, \emph{Generators and relation for $K_2\mathcal{O}_F$}, K-theory, \textbf{31}(2004),195-231.

\bibitem{Bl} S. Bloch, \emph{Higher regulators,algebraic K-theory and zeta functions of elliptic curves}. CRM Monogr Series, vol.11,AMS.,2000.

\bibitem{Bo} A. Borel, \emph{Sur la cohomologie des espaces fibr\'{e}s principaux et des espaces homog\`{e}nes de
groupes de Lie compacts}, Ann. of Math. (2) \textbf{57}(1953),115-207.

\bibitem{Bo2} A. Borel, \emph{Cohomologie de $SL_n$ et valeurs de fonctions z\^{e}ta aux points entiers,}Ann.
Scuola Norm. Sup. Pisa Cl. Sci. (4) \textbf{4} (1977), 613-636; Errata, 7, (1980), 373.

\bibitem{B.J1} J. Browkin,\emph{ Construction of elements in Bloch group}, unpublished, 2012.

\bibitem{B.J2} J. Browkin, H. Gangl, \emph{Tame kernels  and second regulators of number fields and their subfields}, J.K-Theory,\textbf{12}(2013),137-165.

\bibitem{B.J3} J. Browkin, H. Gangl, \emph{Tame and wild kernels of quadratic imaginary number fields }, Math. Comp.\textbf{68}(1999), no.225, 291¨C305.

\bibitem{DB} D.Burns, R.de Jeu, H.Gangl, \emph{On special elememts in higher algebraic K-theory and the Lichtenbaum-Gross Conjecture}.Adv. Math. \textbf{230}(2012), 1502-1529.
\bibitem{C} J.COates,A. Raghuram,Anupam Saikia, R. Sujatha,\emph{The Bloch-Kato Conjecture for the Riemann Zeta Function}. London Math. Soc. Lecture Notes 418, 2015.

\bibitem{Go} A. B. Goncharov, \emph{Geometry of configuration, polylogarithms, and motivic cohomology}, Adv. Math. \textbf{114}(1995),197-318.

\bibitem{Gr} B. H. Gross, \emph{On the values of Artin L-functions}, unpublished, 1980.


\bibitem{KNF} M. Kolster, T. Nguyen Quang Do, V. Fleckinger, \emph{Twisted S-units, p-adic class number formulas, and the Lichtenbaum conjectures}, Duke Math. J., \textbf{84} (1996), 679-717; errata90 (1997), 641-643.

\bibitem{Li}S. Lichtenbaum, \emph{Values of zeta-functions, \'{e}tale cohomology, and algebraic K-theory}, Lecture Notes in Math. \textbf{342} (1973), 489-501.

\bibitem{MW} B. Mazur, A. Wiles, \emph{Class fields of abelian extensions of $\mathbb{Q}$}, Invent. Math. \textbf{76}, no. 2(1984), 179-330.

\bibitem{R.J} J. Rognes and C. Weibel, \emph{Two-primary algebraic K-theory of rings of integers in number fields}.

\bibitem{S.A} A. A. Suslin,  \emph{$K_3$ of a field and the Bloch group}. Translated in Proc. Steklov Inst. Math. 1991, no. 4,217-239.

\bibitem{W.C} C. Weibel, \emph{Algebraic K-theory of rings of integers in local and global fields}, Handbook of K-theory (2005),139-190.

\bibitem{W.L} L. Washington, \emph{Introduction to cyclotomic fields}. Graduate text in mathematic 83, Springer, 1982.

\bibitem{Wi} A. Wiles, \emph{The Iwasawa conjecture for totally real fields}, Ann. of Math.(2)\textbf{131}(1990),493-540.

\bibitem{Zh} L. Zhang, K. Xu, \emph{The tame kernel of $\mathbb{Q}(\zeta_5)$ is trivial}. Math. Comp. \textbf{85}(2016), no.299, 1523-1538.


\end{thebibliography}

\end{document}